\documentclass[a4paper]{amsart}%
\usepackage{amscd}
\usepackage{amsthm}
\usepackage{graphicx}
\usepackage{amsmath}
\usepackage{amsfonts}
\usepackage{amssymb}%
\providecommand{\U}[1]{\protect\rule{.1in}{.1in}}

\newtheorem{theorem}{Theorem}

\newtheorem{lemma}[theorem]{Lemma}

\begin{document}

\title[$L^{p}-L^{q}$ Estimates for Some Convolution Operators]{$L^{p}-L^{q}$ Estimates for Some Convolution Operators with Singular
Measures on The Heisenberg Group}
\author{Pablo Rocha and Tomas Godoy}
\address{Facultad de Matem\'{a}tica, Astronom\'{\i}a y F\'{\i}sica - Ciem
Universidad Nacional de C\'{o}rdoba - Conicet
Ciudad Universitaria, 5000 C\'{o}rdoba, Argentina
}
\email{godoy@famaf.unc.edu.ar, rp@famaf.unc.edu.ar}
\thanks{\textbf{Key words and phrases}: Singular measures, group Fourier transform,
Heisenberg group, convolution operators, spherical transform.}
\thanks{\textbf{2.010 Math. Subject Classification}: 43A80, 42A38.}
\thanks{Partially supported by Agencia Cordoba Ciencia, Secyt-UNC , Conicet and
ANPCYT}
\maketitle

\begin{abstract}
We consider the Heisenberg group $\mathbb{H}^{n}=\mathbb{C}^{n}\times
\mathbb{R}$. Let $\nu $ be the Borel measure on $\mathbb{H}^{n}$ defined by $%
\nu (E)=\int_{\mathbb{C}^{n}}\chi _{E}\left( w,\varphi (w)\right) \eta (w)dw$%
, where $\varphi (w)=\sum\limits_{j=1}^{n}a_{j}\left\vert w_{j}\right\vert
^{2}$, $w=(w_{1},...,w_{n})\in \mathbb{C}^{n}$, $a_{j}\in \mathbb{R}$,
and $\eta (w)=\eta _{0}\left( \left\vert w\right\vert
^{2}\right) $ with $\eta _{0}\in C_{c}^{\infty }(\mathbb{R})$. In this paper
we characterize the set of pairs $(p,q)$ such that the convolution operator
with $\nu $ is $L^{p}(\mathbb{H}^{n})-L^{q}(\mathbb{H}^{n})$ bounded. We
also obtain $L^{p}$-\textit{improving} properties of measures supported on
the graph of the function $\varphi (w)=|w|^{2m}$.
\end{abstract}

\section{Introduction}

Let $\mathbb{H}^{n}=\mathbb{C}^{n}\times \mathbb{R}$ be the
Heisenberg group with group law $\left( z,t\right) \cdot \left( w,s\right)
=\left( z+w,t+s+\left\langle z,w\right\rangle \right) $ where $\langle
z,w\rangle =\frac{1}{2}Im(\sum\limits_{j=1}^{n}z_{j}\cdot \overline{w_{j}})$%
. For $x=(x_{1},...,x_{2n})\in \mathbb{R}^{2n}$, we write $x=(x^{\prime
},x^{\prime \prime })$ with $x^{\prime }\in \mathbb{R}^{n}$, $x^{\prime
\prime }\in \mathbb{R}^{n}$. So, $\mathbb{R}^{2n}$ can be identified with $%
\mathbb{C}^{n}$ via the map $\Psi (x^{\prime },x^{\prime \prime })=x^{\prime
}+ix^{\prime \prime }$. In this setting the form $\langle z,w\rangle $
agrees with the standard symplectic form on $\mathbb{R}^{2n}$. Thus $\mathbb{%
H}^{n}$ can be viewed as $\mathbb{R}^{2n}\times \mathbb{R}$ endowed with the
group law
\[
\left( x,t\right) \cdot \left( y,s\right) =\left( x+y,t+s+\frac{1}{2}%
W(x,y)\right)
\]%
where the symplectic form $W$ is given by $W(x,y)=\sum\limits_{j=1}^{n}%
\left( y_{n+j}x_{j}-y_{j}x_{n+j}\right) $, with $x=(x_{1},...,x_{2n})$ and $%
y=(y_{1},...,y_{2n})$, with neutral element $(0,0)$, and with inverse $%
\left( x,t\right) ^{-1}=\left( -x,-t\right) $.

Let $\varphi :\mathbb{R}^{2n}\rightarrow \mathbb{R}$ be a measurable
function and let $\nu $ be the Borel measure on $\mathbb{H}^{n}$ supported
on the graph of $\varphi $, given by
\begin{equation}
\nu (E)=\int\limits_{\mathbb{R}^{2n}}\chi _{E}\left( w,\varphi \left(
w\right) \right) \eta \left( w\right) dw,  \label{nu}
\end{equation}%
with $\eta (w)=\prod_{j=1}^{n}\eta _{j}\left( \left\vert w_{j}\right\vert
^{2}\right) $, where for $j=1, ..., n$, $\eta _{j}$ is a function in $%
C_{c}^{\infty }(\mathbb{R})$ such that $0\leq \eta _{j}\leq 1$, $\eta
_{j}(t)\equiv 1$ if $t\in \lbrack -1,1]$ and $supp(\eta
_{j})\subset (-2,2)$. Let $T_{\nu }$ be the right convolution operator by $%
\nu $, defined by
\begin{equation}
T_{\nu }f\left( x,t\right) =\left( f\ast \nu \right) \left( x,t\right)
=\int_{\mathbb{R}^{2n}}f\left( \left( x,t\right) \cdot \left( w,\varphi
\left( w\right) \right) ^{-1}\right) \eta \left( w\right) dw.  \label{tnu}
\end{equation}%
We are interested in studying the type set
\[
E_{\nu }=\left\{ \left( \frac{1}{p},\frac{1}{q}\right) \in \left[ 0,1\right]
\times \left[ 0,1\right] :\left\Vert T_{\nu }\right\Vert _{pq}<\infty
\right\}
\]%
where the $L^{p}$ - spaces are taken with respect to the Lebesgue measure on
$\mathbb{R}^{2n+1}$. We say that the measure $\nu $ defined in (\ref{nu}) is
$L^{p}$-\textit{improving} if $E_{\nu }$ does not reduce to the diagonal $%
1/p=1/q$.

This problem is well known if in (\ref{tnu}) we replace the Heisenberg group
convolution with the ordinary convolution in $\mathbb{R}^{2n+1}$. If the
graph of $\varphi $ has non-zero Gaussian curvature at each point, a theorem
of Littman (see \cite{littman}) implies that $E_{\nu }$ is the closed
triangle with vertices $(0,0)$, $(1,1)$, and $\left( \frac{2n+1}{2n+2},\frac{%
1}{2n+2}\right) $ (see \cite{oberlin}). A very interesting survey of results
concerning the type set for convolution operators with singular measures can
be found in \cite{ricci}. Returning to our setting $\mathbb{H}^{n}$, in \cite%
{secco} S. Secco obtains $L^{p}$-\textit{improving} properties of measures
supported on curves in $\mathbb{H}^{1}$, under the assumption that
\begin{eqnarray}
\left\vert
\begin{array}{cc}
\phi _{1}^{(2)} & \phi _{2}^{(2)} \\
\phi _{1}^{(3)} & \phi _{2}^{(3)}%
\end{array}%
\right\vert (s) &\neq &-\frac{(\phi _{1}^{(2)}(s))^{2}}{2},\qquad \forall
s\in I  \nonumber \\
\left\vert
\begin{array}{cc}
\phi _{1}^{(2)} & \phi _{2}^{(2)} \\
\phi _{1}^{(3)} & \phi _{2}^{(3)}%
\end{array}%
\right\vert (s) &\neq &\frac{(\phi _{1}^{(2)}(s))^{2}}{2},\qquad \forall
s\in I  \nonumber
\end{eqnarray}%
where $\Phi (s)=\left( s,\phi _{1}(s),\phi _{2}(s)\right) $ is the curve on
which the measure is supported. In \cite{ricci2} F. Ricci and E. Stein showed that the type set of the
measure given by (\ref{nu}), for the case $\varphi(w)=0$ and $n=1$, is the triangle with vertices
$(0,0),$ $(1,1),$ and $\left( \frac{3}{4},\frac{1}{4}\right)$.

In this article we consider first $\varphi
(w)=\sum\limits_{j=1}^{n}a_{j}\left\vert w_{j}\right\vert ^{2}$, with $%
w_{j}\in \mathbb{R}^{2}$ and $a_{j}\in \mathbb{R}$. The
Riesz-Thorin theorem implies that the type set $E_{\nu }$ is a convex subset
of $[0,1]\times \lbrack 0,1]$. In Lemmas 3 and 4 we obtain the following
necessary conditions on the pairs $\left( \frac{1}{p},\frac{1}{q}\right) \in
E_{\nu },$
\begin{eqnarray*}
\frac{1}{q} &\leq &\frac{1}{p}, \\
\frac{1}{q} &\geq &\frac{2n+1}{p}-2n \\
\frac{1}{q} &\geq &\frac{1}{(2n+1)p}
\end{eqnarray*}%
Thus $E_{\nu }$ is contained in the closed triangle with vertices $(0,0),$ $%
(1,1),$ and $\left( \frac{2n+1}{2n+2},\frac{1}{2n+2}\right) $. In Section 3
we prove that $E_{\nu }$ is exactly the closed triangle with these vertices.
Indeed, we obtain the following

\begin{theorem}
\textit{If }$\nu $ \textit{is the Borel measure defined by (\ref{nu}),
supported on the graph of the function }$\varphi
(w)=\sum\limits_{j=1}^{n}a_{j}\left\vert w_{j}\right\vert ^{2}$, with $%
w_{j}\in \mathbb{R}^{2}$ and $a_{j}\in \mathbb{R}$,\textit{\
then the type set }$E_{\nu }$ \textit{is the closed triangle with vertices}
\[
A=\left( 0,0\right) ,\qquad B=\left( 1,1\right) ,\qquad C=\left( \frac{2n+1}{%
2n+2},\frac{1}{2n+2}\right)
\]%
\textit{with }$n\in \mathbb{N}$.
\end{theorem}

In a similar way we also obtain $L^{p}$-\textit{improving} properties of the
measure supported on the graph of the function $\varphi (w)=|w|^{2m}$. In
fact we prove the following

\begin{theorem}
For $m,n\in \mathbb{N}_{\geq 2}$ let $\nu _{m}$ be the measure given by (\ref%
{nu}) with $\varphi (y)=\left\vert y\right\vert ^{2m},$ $y\in \mathbb{R}%
^{2n}.$ Then the type set $E_{\nu _{m}}$ contains the closed triangle with
vertices $\left( 0,0\right) ,$ $\left( 1,1\right) ,$ $\left( \frac{2(1+mn)-m%
}{2(1+mn)},\frac{m}{2(1+mn)}\right) $.
\end{theorem}

Throughout this work, $c$ will denote a positive constant not necessarily
the same at each occurrence.

\section{Necessary conditions}

We denote $B(r)$ the $2n+1$ dimensional ball centered at the origin with
radius $r.$

\begin{lemma}
\textit{Let }$\nu $ \textit{be the Borel measure defined by (\ref{nu}), where%
} $\varphi $ \textit{is a bounded measurable function. If} $\left( \frac{1}{p%
},\frac{1}{q}\right) \in E_{\nu }$ \textit{then }$p\leq q$.
\end{lemma}

\begin{proof}
For $(y,s) \in \mathbb{H}^{n}$ we define the operator $\tau_{(y,s)}$ by $(\tau_{(y,s)}f)(x,t) = f((y,s)^{-1} \cdot (x,t))$.
Since $\tau_{(y,s)} T_{\nu} = T_{\nu} \tau_{(y,s)}$, it is easy to see that the $\mathbb{R}^{n}$ argument utilized in the proof of Theorem 1.1 in \cite{hor}
works as well on $\mathbb{H}^{n}$.
\end{proof}

\begin{lemma}
\textit{Let }$\nu $ \textit{be the Borel measure defined by (\ref{nu}), where%
} $\varphi $ \textit{is a smooth function.} \textit{Then} $E_{\nu }$ \textit{%
is contained in the closed triangle with vertices}
\[
(0,0)\,,\,(1,1)\,,\,\left( \frac{2n+1}{2n+2},\frac{1}{2n+2}\right) .
\]
\end{lemma}

\begin{proof}
We will prove that if $\left( \frac{1}{p},\frac{1}{q}\right) \in $ $E_{\nu }$
then $\frac{1}{q}\geq \frac{2n+1}{p}-2n$ and $\frac{1}{q}\geq \frac{1}{%
(2n+1)p}$. Then the lemma will follow by the Riesz-Thorin theorem. Let $%
f_{\delta }=\chi _{Q_{\delta }}$, where $Q_{\delta }=B(2\delta ).$ Let $%
D=\left\{ x\in \mathbb{R}^{2n}:\left\Vert x\right\Vert \leq 1\right\} $ and $%
A_{\delta }$ the set defined by
\[
A_{\delta }=\left\{ (x,t)\in \mathbb{R}^{2n}\times \mathbb{R}:x\in D\
;\left\vert t-\varphi (x)\right\vert \leq \frac{\delta }{4}\right\} .
\]%
For each $(x,t)\in A_{\delta }$ fixed, we define $F_{\delta ,x}$ by
\[
F_{\delta ,x}=\left\{ y\in D:\left\Vert x-y\right\Vert _{\mathbb{R}%
^{2n}}\leq \frac{\delta }{4n(1+\left\Vert \nabla \varphi \mid _{supp
(\eta )}\right\Vert _{\infty })}\right\} .
\]%
Now, for each $(x,t)\in A_{\delta }$ fixed, we have
\begin{equation}
(x,t)\cdot (y,\varphi (y))^{-1}\in Q_{\delta },\qquad \forall y\in F_{\delta
,x},  \label{qdelta}
\end{equation}%
indeed
\[
\left\Vert (x,t)\cdot (y,\varphi (y))^{-1}\right\Vert _{\mathbb{R}%
^{2n+1}}\leq \left\Vert x-y\right\Vert _{\mathbb{R}^{n}\times \mathbb{R}%
^{n}}
\]%
\[
+\left\vert t-\varphi (x)\right\vert +\left\vert \varphi (x)-\varphi
(y)\right\vert +\frac{1}{2}\left\vert W(x,y)\right\vert ,
\]%
since
\[
\frac{1}{2}\left\vert W(x,y)\right\vert \leq n\left\Vert x\right\Vert _{%
\mathbb{R}^{2n}}\left\Vert x-y\right\Vert _{\mathbb{R}^{2n}},
\]%
(\ref{qdelta}) follows. Then for $(x,t)\in A_{\delta }$ we obtain
\[
T_{\nu }f_{\delta }(x,t)\geq \int\limits_{F_{\delta ,x}}\eta (y)dy\geq
c\delta ^{2n},
\]%
where $c$ not depends on $\delta $, $x$ and $t$. If $(\frac{1}{p},\frac{1}{q}%
)\in E_{\nu }$ implies
\[
c\delta ^{\frac{1}{q}+2n}=c\delta ^{2n}\left\vert A_{\delta }\right\vert ^{%
\frac{1}{q}}\leq \left( \int\limits_{A_{\delta }}\left\vert T_{\nu
}f_{\delta }(x,t)\right\vert ^{q}\right) ^{\frac{1}{q}}\leq \left\Vert
T_{\nu }f_{\delta }\right\Vert _{q}\leq c_{p,q}\left\Vert f_{\delta
}\right\Vert _{p}=c\delta ^{\frac{2n+1}{p}},
\]%
thus $\delta ^{2n+\frac{1}{q}}\leq C\delta ^{\frac{2n+1}{p}}$ for all $%
0<\delta <1$ small enough. This implies that
\[
\frac{1}{q}\geq \frac{2n+1}{p}-2n.
\]%
Now, the adjoint operator of $T_{\nu}$ is given by
$$T^{*}_{\nu}g(x,t)=\int_{\mathbb{R}^{2n}} \, g \left( (x,t) \cdot (y, \varphi(y))\right) \eta(y) \, dy$$
and let $E^{*}_{\nu}$ be the type set corresponding. Since $T_{\nu}=(T^{*}_{\nu})^{*}$, by duality it follows that
$\left(\frac{1}{p}, \frac{1}{p'} \right) \in E_{\nu}$ if and only if $\left(\frac{1}{p}, \frac{1}{p'} \right) \in E^{*}_{\nu}$,
thus if $\left( \frac{1}{p}, \frac{1}{q} \right) \in E^{*}_{\nu}$ then $\frac{1}{q}\geq \frac{2n+1}{p}-2n$.
Finally, by duality it is also necessary that
\[
\frac{1}{q}\geq \frac{1}{(2n+1)p}.
\]%
Therefore $E_{\nu }$ is contained in the region determined by these two
conditions and by the condition $p\leq q$, i.e.: the closed triangle with
vertices $(0,0)$, $(1,1)$, $(\frac{2n+1}{2n+2},\frac{1}{2n+2})$.
\end{proof}

\qquad

\textbf{Remark}
\textit{Lemma 4 holds if we replace the smoothness condition with a Lipschitz
condition.}

\section{The Main Results}

We consider for each $N\in \mathbb{N}$ fixed, an auxiliary operator $T_{N}$
which will be embedded in a analytic family of operators $\left\{
T_{N,z}\right\} $ on the strip $-n\leq Re(z)\leq 1$ such that
\begin{equation}
\left\{
\begin{array}{c}
\left\Vert T_{N,z}\left( f\right) \right\Vert _{L^{\infty }\left( \mathbb{H}%
^{n}\right) }\leq c_{z}\left\Vert f\right\Vert _{L^{1}\left( \mathbb{H}%
^{n}\right) }\qquad Re(z)=1 \\
\left\Vert T_{N,z}\left( f\right) \right\Vert _{L^{2}\left( \mathbb{H}%
^{n}\right) }\leq c_{z}\left\Vert f\right\Vert _{L^{2}\left( \mathbb{H}%
^{n}\right) }\qquad Re(z)=-n%
\end{array}%
\right\}   \label{prop}
\end{equation}%
where $c_{z}$ will depend admissibly on the variable $z$ and it will not
depend on $N.$ We denote $T_{N}=T_{N,0}$. By Stein's theorem of complex
interpolation, it will follow that the operator $T_{N}$ will be bounded from
$L^{\frac{2n+2}{2n+1}}(\mathbb{H}^{n})$ in $L^{2n+2}(\mathbb{H}^{n})$
uniformly on $N$, if we see that $T_{N}f\left( x,t\right) \rightarrow T_{\nu
}f\left( x,t\right) $ as $N\rightarrow \infty ,$ a.e $(x,t)\in \mathbb{R}%
^{2n+1}$. Theorem 1 will then follow from Fatou's lemma and the lemmas 3 and
4. To prove the second inequality in (\ref{prop}) we will see that such
family will admit the following expression
\[
T_{N,z}(f)(x,t)=\left( f\ast K_{N,z}\right) (x,t),
\]%
where $K_{N,z}\in L^{1}(\mathbb{H}^{n})$, moreover it is a poliradial
function (i.e. the values of $K_{N,z}$ depend on $\left\vert
w_{1}\right\vert ,$...$,\left\vert w_{n}\right\vert $ and $t$). Now our
operator $T_{N,z}$ can be realized as a multiplication of operators via the
group Fourier transform, i.e.
\[
\widehat{T_{N,z}(f)}(\lambda )=\widehat{f}(\lambda )\widehat{K_{N,z}}%
(\lambda )
\]%
where, for each $\lambda \neq 0$, $\widehat{K_{N,z}}(\lambda )$ is an
operator on the Hilbert space $L^{2}(\mathbb{R}^{n})$ given by
\[
\widehat{K_{N,z}}(\lambda )g(\xi )=\int\limits_{\mathbb{H}%
^{n}}K_{N,z}(\varsigma ,t)\pi _{\lambda }(\varsigma ,t)g(\xi )d\varsigma dt.
\]%
It then follows from Plancherel's theorem for the group Fourier transform
that
\[
\left\Vert T_{N,z}f\right\Vert _{L^{2}(\mathbb{H}^{n})}\leq A_{z}\left\Vert
f\right\Vert _{L^{2}(\mathbb{H}^{n})}
\]%
if and only if
\begin{equation}
\left\Vert \widehat{K_{N,z}}(\lambda )\right\Vert _{op}\leq A_{z}
\label{L21}
\end{equation}%
uniformly over\textit{\ }$N$\textit{\ }and\textit{\ }$\lambda \neq 0.$ Since
$K_{N,z}$ is a poliradial integrable function, then by a well known result
of Geller (see Lemma 1.3, p. 213 in \cite{geller}), the operators $\widehat{%
K_{N,z}}(\lambda ):L^{2}(\mathbb{H}^{n})\rightarrow L^{2}(\mathbb{H}^{n})$
are, for each $\lambda \neq 0$, diagonal with respect to a Hermite basis for
$L^{2}(\mathbb{R}^{n})$. This is
\[
\widehat{K_{N,z}}(\lambda )=C_{n}\left( \delta _{\gamma ,\alpha }\mu
_{N,z}(\alpha ,\lambda )\right) _{\gamma ,\alpha \in \mathbb{N}_{0}^{n}}
\]%
where $C_{n}=(2\pi )^{n}$, $\alpha =(\alpha _{1},...,\alpha _{n})$, $\delta _{\gamma ,\alpha }=1$ if $\gamma = \alpha$ and $\delta _{\gamma ,\alpha }=0$ if $\gamma \neq \alpha$, and the
diagonal entries $\mu _{N,z}(\alpha _{1},...,\alpha _{n},\lambda )$ can be
expressed explicitly in terms of the Laguerre transform. We have in fact
\[
\mu _{N,z}(\alpha _{1},...,\alpha _{n},\lambda )=\int\limits_{0}^{\infty
}\,...\,\int\limits_{0}^{\infty }\,K_{N,z}^{\lambda
}(r_{1},...,r_{n})\prod_{j=1}^{n}\left( r_{j}L_{\alpha _{j}}^{0}(\frac{1}{2}%
\left\vert \lambda \right\vert r_{j}^{2})e^{-\frac{1}{4}\left\vert \lambda
\right\vert r_{j}^{2}}\right) \,dr_{1}...dr_{n}
\]%
where $L_{k}^{0}(s)$ are the Laguerre polynomials, i.e. $L_{k}^{0}(s)=%
\sum_{i=0}^{k}\left( \frac{k!}{(k-i)!i!}\right) \frac{(-s)^{i}}{i!}$ and $%
K_{N,z}^{\lambda }(\varsigma )=\int\limits_{\mathbb{R}}K_{N,z}(\varsigma
,t)e^{i\lambda t}dt.$ Now (\ref{L21}) is equivalent to
\[
\left\Vert T_{N,z}f\right\Vert _{L^{2}(\mathbb{H}^{n})}\leq A_{z}\left\Vert
f\right\Vert _{L^{2}(\mathbb{H}^{n})}
\]%
if and only if
\begin{equation}
\left\vert \mu _{N,z}(\alpha _{1},...,\alpha _{n},\lambda )\right\vert \leq
A_{z}  \label{L22}
\end{equation}%
uniformly over\textit{\ }$N$, $\alpha _{j}$\textit{\ }and\textit{\ }$\lambda
\neq 0.$ If $Re(z)=-n$ we prove that $\left\vert \mu _{N,z}(\alpha
_{1},...,\alpha _{n},\lambda )\right\vert \leq A_{z}$, with $A_{z}$
independent of $N$, $\lambda \neq 0$ and $\alpha _{j}$, and then we obtain
the boundedness on $L^{2}(\mathbb{H}^{n})$ that is stated in (\ref{prop}).

We consider the family $\{ I_{z} \}_{z \in \mathbb{C}}$ of distributions on $\mathbb{R}$ that arises by analytic continuation of the family
$\{ I_{z} \}$ of functions, initially given when $Re(z)>0$ and $s\in \mathbb{R} \setminus\{ 0 \}$ by
\begin{equation}
I_{z}(s)=\frac{2^{-\frac{z}{2}}}{\Gamma \left( \frac{z}{2}\right) }%
\left\vert s\right\vert ^{z-1}.  \label{iz}
\end{equation}%
In particular, we have $%
\widehat{I_{z}}=I_{1-z}$, also $I_{0}=c\delta $ where $\widehat{\cdot }$
denotes the Fourier transform on $\mathbb{R}$ and $\delta $ is the Dirac
distribution at the origin on $\mathbb{R}$.

Let $H\in S(\mathbb{R)}$ such that $supp(\widehat{H})\subseteq
\left( -1,1\right) $ and $\int \widehat{H}(t)dt=1$. Now we put $\phi
_{N}(t)=H(\frac{t}{N})$ thus $\widehat{\phi _{N}}(\xi )=N\widehat{H}(N\xi )$
and $\widehat{\phi _{N}}\rightarrow \delta $ in the sense of the
distribution, as $N\rightarrow \infty $.

For $z\mathbb{\in C}$ and $N\in \mathbb{N}$, we also define $J_{N,z}$ as the
distribution on $\mathbb{H}^{n}$ given by the tensor products
\begin{equation}
J_{N,z}=\delta \otimes ...\otimes \delta \otimes \left( I_{z}\ast _{\mathbb{R%
}}\widehat{\phi _{N}}\right)  \label{jz}
\end{equation}%
where $\ast _{\mathbb{R}}$ denotes the usual convolution on $\mathbb{R}$ and
$I_{z}$ is the fractional integration kernel given by (\ref{iz}). Finally,
for $z\in \mathbb{C}$ and $N\in \mathbb{N}$ fixed, we defined the operator $%
T_{N,z}$ by
\begin{equation}
T_{N,z}f(x,t)=\left( f\ast \nu \ast J_{N,z}\right) (x,t)  \label{tz}
\end{equation}%
We observe that $T_{N,0}f(x,t)\rightarrow cT_{\nu }f(x,t)$ as $N\rightarrow
\infty $ a.e $(x,t)\in \mathbb{R}^{2n+1},$ since $J_{N,0}=\delta \otimes
...\otimes \delta \otimes c\widehat{\phi _{N}}\rightarrow \delta \otimes
...\otimes \delta \otimes c\delta $ in the sense of the distribution, as $%
N\rightarrow \infty $.

\qquad

Before proving Theorem 1 we need the following lemmas,

\begin{lemma}
\textit{If }$Re(z)\leq -1$ \textit{then }$\nu \ast J_{N,z}\in L^{p}(\mathbb{H%
}^{n}),$ $\forall p\geq 1.$
\end{lemma}

\begin{proof}
For $Re(z)\leq -1$ and $N \in \mathbb{N}$ fixed, a simple calculation gives
\[
\left( \nu \ast J_{N,z}\right) (x,\sigma )=\eta (x)\left( I_{z}\ast _{%
\mathbb{R}}\widehat{\phi _{N}}\right) (\sigma -\varphi (x)).
\]%
We see
that is enough to prove that $\left( I_{z}\ast \widehat{\phi _{N}}\right) (s) \in L^{p}(\mathbb{R})$, if $Re(z)\leq
-1 $.
For them we observe that if $g \in \mathcal{S}(\mathbb{R})$ with $supp(g) \cap [-\epsilon,\epsilon] = \emptyset$ for some $\epsilon > 0$, then for $Re(z)\leq -1$
$$I_{z}(g) = \frac{2^{-\frac{z}{2}}}{\Gamma \left( \frac{z}{2}\right) } \int_{|t|\geq \epsilon} \left\vert t \right\vert ^{z-1} g(t) dt, \,\,\,\,\,\,
if \,\,\, z \notin -2\mathbb{N}$$ and $$I_{z}(g) = 0, \,\,\,\,\,\, if \,\,\, z \in -2\mathbb{N}.$$
From this observation and the fact that $$supp \left( \tau_{s} \left( \widehat{ \phi _{N}} ^{\vee} \right) \right) \subset \left[s - \frac{1}{N}, s + \frac{1}{N} \right] \subset [-\infty,-1] \cup [1,+\infty] \,\,\,\,\, for \,\,\, |s| \geq \frac{N+1}{N}$$ (where $\phi ^{\vee }(x)=\phi (-x)$ and $(\tau_{s}\phi)(x) = \phi(x-s)$), we obtain $$\left|\left( I_{z}\ast \widehat{\phi _{N}}\right) (s) \right| = \left| I_{z} \left( \tau_{s} \left( \widehat{ \phi _{N}} ^{\vee} \right) \right) \right|  \leq c \left|s - \frac{sign(s)}{N} \right|^{-2}, \,\,\,\,\, if \,\,\, |s| \geq \frac{N+1}{N}.$$ Finally, since $\left|\left( I_{z}\ast \widehat{\phi _{N}}\right) (s) \right| \leq c$ for all
$s \in [-2,2]$, the lemma follows.
\end{proof}

\begin{lemma}
\textit{For }$n\in \mathbb{N}$ \textit{and }$k\in \mathbb{N}_{0}$ \textit{we
put }
\[
F_{n,k}(\sigma ):=\chi _{(0,\infty )}(\sigma )L_{k}^{n-1}\left( \sigma
\right) e^{-\frac{\sigma }{2}}\sigma ^{n-1}
\]%
\textit{then}
\[
\widehat{F_{n,k}}(\xi )=\frac{(k+n-1)!}{k!}\frac{\left( -\frac{1}{2}+i\xi
\right) ^{k}}{\left( \frac{1}{2}+i\xi \right) ^{k+n}}.
\]
\end{lemma}

\begin{proof}
On $\mathbb{R}$ we define the Fourier transform by $\widehat{g}(\xi )=\int_{%
\mathbb{R}}\,g(\sigma )e^{-i\sigma \xi }\,d\sigma $ thus
\[
\widehat{F_{n,k}}(\xi )=\int\limits_{0}^{\infty }L_{k}^{n-1}\left( \sigma
\right) \sigma ^{n-1}e^{-\sigma \left( \frac{1}{2}+i\xi \right) }d\sigma
\]%
and since $L_{k}^{n-1}\left( \sigma \right) \sigma ^{n-1}=\frac{e^{\sigma }}{%
k!}\left( \frac{d}{d\sigma }\right) ^{k}\left( e^{-\sigma }\sigma
^{k+n-1}\right) $ for each $n\in \mathbb{N}$ and each $k\in \mathbb{N}_{0}$,
we obtain
\begin{eqnarray*}
\widehat{F_{n,k}}(\xi ) &=&\frac{1}{k!}\int\limits_{0}^{\infty }\left( \frac{%
d}{d\sigma }\right) ^{k}\left( e^{-\sigma }\sigma ^{k+n-1}\right) e^{-\sigma
\left( -\frac{1}{2}+i\xi \right) }d\sigma  \\
&=&\frac{\left( -\frac{1}{2}+i\xi \right) ^{k}}{k!}\int\limits_{0}^{\infty
}\sigma ^{k+n-1}e^{-\sigma \left( \frac{1}{2}+i\xi \right) }d\sigma  \\
&=&\frac{\left( -\frac{1}{2}+i\xi \right) ^{k}}{k!}\int\limits_{0}^{\infty }%
\frac{s^{k+n-1}}{(\frac{1}{2}+i\xi )^{k+n-1}}e^{-s}\frac{ds}{(\frac{1}{2}%
+i\xi )} \\
&=&\frac{(k+n-1)!}{k!}\frac{\left( -\frac{1}{2}+i\xi \right) ^{k}}{\left(
\frac{1}{2}+i\xi \right) ^{k+n}}
\end{eqnarray*}%
the third equality follows from the rapid decay of the function $e^{-z}$ on
the region $\left\{ z:Re(z)>0\right\} .$ Then we apply the Cauchy's theorem.
\end{proof}

\qquad

\textbf{Proof of Theorem 1. }For $Re(z)=1$ we have
\[
\left\Vert T_{N,z}f\right\Vert _{\infty }=\left\Vert \left( f\ast \nu \ast
J_{N,z}\right) \right\Vert _{\infty }\leq \left\Vert f\right\Vert
_{1}\left\Vert \nu \ast J_{N,z}\right\Vert _{\infty }
\]%
Since
\[
\left( \nu \ast J_{N,z}\right) (x,\sigma )=\eta (x)\left( I_{z}\ast _{%
\mathbb{R}}\widehat{\phi _{N}}\right) (\sigma -\varphi \left( x\right) )
\]%
it follows that $\left\Vert \nu \ast J_{N,z}\right\Vert _{\infty }\leq
c\left\vert \Gamma \left( \frac{z}{2}\right) \right\vert ^{-1}$. Then, for $%
Re(z)=1$, we obtain
\[
\left\Vert T_{N,z}\right\Vert _{1,\infty }\leq c\left\vert \Gamma \left(
\frac{z}{2}\right) \right\vert ^{-1}.
\]

From Lemma 5, in particular, we have that $\nu \ast J_{N,z}\in L^{1}(\mathbb{%
H}^{n})\cap L^{2}(\mathbb{H}^{n}).$ In addition $\nu \ast J_{N,z}$ is a
poliradial function. Thus the operator $\left( \nu \ast J_{N,z}\right)
\widehat{\text{ \ \ }}(\lambda )$ is diagonal with respect to a Hermite base
for $L^{2}(\mathbb{R}^{n})$, and its diagonal entries $\mu _{N,z}(\alpha
,\lambda )$, with $\alpha =(\alpha _{1},...,\alpha _{n})\in \mathbb{N}%
_{0}^{n},$ are given by
\[
\mu _{N,z}(\alpha ,\lambda )=\int\limits_{0}^{\infty
}...\int\limits_{0}^{\infty }\,\left( \nu \ast J_{N,z}\right)
(r_{1},...,r_{n},\widehat{-\lambda })\prod_{j=1}^{n}\left(
r_{j}L_{k_{j}}^{0}(\frac{1}{2}\left\vert \lambda \right\vert r_{j}^{2})e^{-%
\frac{1}{4}\left\vert \lambda \right\vert r_{j}^{2}}\right) \,dr_{1}...dr_{n}
\]%
\[
=\int\limits_{0}^{\infty }\,...\,\int\limits_{0}^{\infty }\,\left( I_{z}\ast
_{\mathbb{R}}\widehat{\phi _{N}}\right) \widehat{\text{ \ \ }}(-\lambda
)\prod_{j=1}^{n}\left( \eta _{j}(r_{j}^{2})e^{i\lambda
a_{j}r_{j}^{2}}r_{j}L_{k_{j}}^{0}(\frac{1}{2}\left\vert \lambda \right\vert
r_{j}^{2})e^{-\frac{1}{4}\left\vert \lambda \right\vert r_{j}^{2}}\right)
\,dr_{1}...dr_{n}
\]%
\[
=I_{1-z}(-\lambda )\phi _{N}(\lambda )\prod_{j=1}^{n}\int\limits_{0}^{\infty
}\,\left( \eta _{j}(r_{j}^{2})e^{i\lambda a_{j}r_{j}^{2}}r_{j}L_{\alpha
_{j}}^{0}(\frac{1}{2}\left\vert \lambda \right\vert r_{j}^{2})e^{-\frac{1}{4}%
\left\vert \lambda \right\vert r_{j}^{2}}\right) \,dr_{j}.
\]%
Thus, it is enough to study the integral $\int\limits_{0}^{\infty }\eta
_{1}(r^{2})L_{\alpha _{1}}^{0}\left( \frac{\left\vert \lambda \right\vert
r^{2}}{2}\right) e^{-\frac{\left\vert \lambda \right\vert r^{2}}{4}%
}e^{i\lambda a_{1}r^{2}}r\,dr$, where $a_{1}\in \mathbb{R}$
and $\eta _{1}\in C_{c}^{\infty }(\mathbb{R})$. We make the change of
variable $\sigma =\frac{\left\vert \lambda \right\vert r^{2}}{2}$ in such
integral and we obtain
\[
\int\limits_{0}^{\infty }\eta _{1}(r^{2})L_{\alpha _{1}}^{0}\left( \frac{%
\left\vert \lambda \right\vert r^{2}}{2}\right) e^{-\frac{\left\vert \lambda
\right\vert r^{2}}{4}}e^{i\lambda a_{1}r^{2}}r\,dr
\]%
\[
=\left\vert \lambda \right\vert ^{-1}\int\limits_{0}^{\infty }\eta
_{1}\left( \frac{2\sigma }{\left\vert \lambda \right\vert }\right) L_{\alpha
_{1}}^{0}\left( \sigma \right) e^{-\frac{\sigma }{2}}e^{i2sgn(\lambda
)a_{1}\sigma }\,d\sigma
\]%
\[
=\left\vert \lambda \right\vert ^{-1}\left( F_{\alpha _{1}}G_{\lambda
}\right) \widehat{\left. {}\right. }(-2sgn(\lambda )a_{1})=\left\vert
\lambda \right\vert ^{-1}(\widehat{F_{\alpha _{1}}}\ast \widehat{G_{\lambda }%
})(-2sgn(\lambda )a_{1})
\]%
where
\begin{equation}
F_{\alpha _{1}}(\sigma ):=\chi _{(0,\infty )}(\sigma )L_{\alpha
_{1}}^{0}\left( \sigma \right) e^{-\frac{\sigma }{2}}  \nonumber
\end{equation}%
and
\begin{equation}
G_{\lambda }(\sigma ):=\eta _{1}\left( \frac{2\sigma }{\left\vert \lambda
\right\vert }\right)   \nonumber
\end{equation}%
Now
\[
\left\vert (\widehat{F_{\alpha _{1}}}\ast \widehat{G_{\lambda }}%
)(-2sgn(\lambda )a_{1})\right\vert \leq \left\Vert \widehat{F_{\alpha _{1}}}%
\ast \widehat{G_{\lambda }}\right\Vert _{\infty }\leq \left\Vert \widehat{%
F_{\alpha _{1}}}\right\Vert _{\infty }\left\Vert \widehat{G_{\lambda }}%
\right\Vert _{1}=\left\Vert \widehat{F_{\alpha _{1}}}\right\Vert _{\infty
}\left\Vert \widehat{\eta _{1}}\right\Vert _{1}.
\]%
So it is enough to estimate $\left\Vert \widehat{F_{\alpha _{1}}}\right\Vert
_{\infty }$. Now, from lemma 6, with $n=1$ and $k=\alpha _{1}$, we obtain
\begin{equation}
\left\vert \widehat{F_{\alpha _{1}}}(\xi )\right\vert =\frac{1}{\left\vert
\frac{1}{2}+i\xi \right\vert }  \nonumber
\end{equation}%
Finally, for $Re(z)=-n$, we obtain
\begin{eqnarray*}
\left\vert \mu _{N,z}(\alpha _{1},...,\alpha _{n},\lambda )\right\vert
&\leq &2^{n}\left\vert I_{1-z}(-\lambda )\phi _{N}(\lambda )\right\vert
\left\vert \lambda \right\vert ^{-n}\prod_{j=1}^{n}\Vert \widehat{\eta _{j}}%
\Vert _{1} \\
&\leq &2^{n}\left\vert \Gamma \left( \frac{1-z}{2}\right) \right\vert
^{-1}\left\vert H(\frac{\lambda }{N})\right\vert \prod_{j=1}^{n}\Vert
\widehat{\eta _{j}}\Vert _{1} \\
&\leq &2^{n}\left\vert \Gamma \left( \frac{1-z}{2}\right) \right\vert
^{-1}\left\Vert H\right\Vert _{\infty }\prod_{j=1}^{n}\Vert \widehat{\eta
_{j}}\Vert _{1}
\end{eqnarray*}%
by (\ref{L22}) it follows, for $Re(z)=-n$, that
\[
\left\Vert T_{N,z}f\right\Vert _{L^{2}(\mathbb{H}^{n})}\leq c\frac{(2\pi
)^{n}2^{n}}{\left\vert \Gamma \left( \frac{1-z}{2}\right) \right\vert }%
\left\Vert f\right\Vert _{L^{2}(\mathbb{H}^{n})}
\]%
It is easy to see, with the aid of the Stirling formula (see \cite{stein4},
p. 326), that the family $\left\{ T_{N,z}\right\} $ satisfies, on the strip $%
-n\leq Re(z)\leq 1$, the hypothesis of the complex interpolation theorem
(see \cite{stein2} p. 205) and so $T_{N,0}$ is bounded from $L^{\frac{2n+2}{%
2n+1}}(\mathbb{H}^{n})$ into $L^{2n+2}(\mathbb{H}^{n})$ uniformly on $N$,
then doing $N$ tend to infinity, we obtain that the operator $T_{\nu }$ is
bounded from $L^{\frac{2n+2}{2n+1}}(\mathbb{H}^{n})$ into $L^{2n+2}(\mathbb{H%
}^{n})$ with $n\in \mathbb{N}$.$\blacksquare $

\qquad

\textbf{Proof of Theorem 2.} We consider for each $N\in \mathbb{N}$ fixed,
the analytic family of operators $\left\{ U_{N,z}\right\} $ on the strip $%
-\left( n+\frac{1-m}{m}\right) \leq Re(z)\leq 1$, defined by $U_{N,z}f=f\ast
\nu _{m}\ast J_{N,z}$, where $J_{N,z}$ is given by (\ref{jz}) and $%
U_{N,0}f\rightarrow U_{\nu _{m}}f=f\ast \nu _{m}$ as $N\rightarrow \infty $.
Proceeding as in proof of Theorem 1 it follows, for $Re(z)=1$, that $%
\left\Vert U_{N,z}\right\Vert _{1,\infty }\leq c\left\vert \Gamma \left(
\frac{z}{2}\right) \right\vert ^{-1}$. Also it is clear that, for $%
Re(z)=-\left( n+\frac{1-m}{m}\right) $, the kernel $\nu _{m}\ast J_{N,z}\in
L^{1}(\mathbb{H}^{n})\cap L^{2}(\mathbb{H}^{n})$ and it is also a radial
function. Now, our operator $\left( \nu _{m}\ast J_{N,z}\right) \widehat{%
\left. {}\right. }(\lambda )$ is diagonal, with diagonal entries $\upsilon
_{N,z}(k,\lambda )$ given by
\begin{eqnarray*}
\upsilon _{N,z}(k,\lambda ) &=&\frac{k!}{(k+n-1)!}\int\limits_{0}^{\infty
}\left( \nu _{m}\ast J_{N,z}\right) (s,\widehat{-\lambda })L_{k}^{n-1}\left(
\frac{\left\vert \lambda \right\vert s^{2}}{2}\right) e^{-\frac{\left\vert
\lambda \right\vert s^{2}}{4}}s^{2n-1}ds
\end{eqnarray*}%
$$=\frac{k!}{(k+n-1)!}I_{1-z}(-\lambda )\phi _{N}(\lambda
)\int\limits_{0}^{\infty }\eta _{0}(s^{2})L_{k}^{n-1}\left( \frac{\left\vert
\lambda \right\vert s^{2}}{2}\right) e^{-\frac{\left\vert \lambda
\right\vert s^{2}}{4}}e^{i\lambda s^{2m}}s^{2n-1}ds
$$
Now we study the integral
\[
\int\limits_{0}^{\infty }\eta _{0}(s^{2})L_{k}^{n-1}\left( \frac{\left\vert
\lambda \right\vert s^{2}}{2}\right) e^{-\frac{\left\vert \lambda
\right\vert s^{2}}{4}}e^{i\lambda s^{2m}}s^{2n-1}ds.
\]%
We make the change of variable $\sigma =\frac{\left\vert \lambda \right\vert
s^{2}}{2}$ to obtain
\[
\int\limits_{0}^{\infty }\eta _{0}(s^{2})L_{k}^{n-1}\left( \frac{\left\vert
\lambda \right\vert s^{2}}{2}\right) e^{-\frac{\left\vert \lambda
\right\vert s^{2}}{4}}e^{i\lambda s^{2m}}s^{2n-1}ds
\]%
\[
=2^{n-1}\left\vert \lambda \right\vert ^{-n}\int\limits_{0}^{\infty }\eta
_{0}\left( \frac{2\sigma }{\left\vert \lambda \right\vert }\right)
L_{k}^{n-1}\left( \sigma \right) e^{-\frac{\sigma }{2}}e^{i2^{m}sgn(\lambda
)\left\vert \lambda \right\vert ^{1-m}\sigma ^{m}}\sigma ^{n-1}d\sigma
\]%
\[
=2^{n-1}\left\vert \lambda \right\vert ^{-n}\left( F_{n,k}G_{\lambda
}R_{\lambda }\right) \widehat{\left. {}\right. }(0)=2^{n-1}\left\vert
\lambda \right\vert ^{-n}(\widehat{F_{n,k}}\ast \widehat{G_{\lambda
}R_{\lambda }})(0)
\]%
\[
=2^{n-1}\left\vert \lambda \right\vert ^{-n}\left( \widehat{F_{n,k}}\ast
\left( \widehat{G_{\lambda }}\ast \widehat{R_{\lambda }}\right) \right) (0)
\]%
where $F_{n,k}$ is the function defined in the lemma 6, $G_{\lambda }(\sigma
)=\eta _{0}\left( 2\sigma /|\lambda |\right) $ and $R_{\lambda }(\sigma
)=\chi _{(0,\left\vert \lambda \right\vert )}(\sigma )e^{i2^{m}sgn(\lambda
)\left\vert \lambda \right\vert ^{1-m}\sigma ^{m}}$. If $n\geq 2$, from
lemma 6 we get
\begin{eqnarray*}
\left\Vert \widehat{F_{n,k}}\ast \left( \widehat{G_{\lambda }}\ast \widehat{%
R_{\lambda }}\right) \right\Vert _{\infty } &\leq &\left\Vert \widehat{%
F_{n,k}}\right\Vert _{1}\left\Vert \widehat{G_{\lambda }}\right\Vert
_{1}\left\Vert \widehat{R_{\lambda }}\right\Vert _{\infty } \\
&=&\frac{(k+n-1)!}{k!}\left( \int\limits_{\mathbb{R}}\frac{d\xi }{\left(
\frac{1}{4}+\xi ^{2}\right) ^{\frac{n}{2}}}\right) \left\Vert \widehat{\eta
_{0}}\right\Vert _{1}\left\Vert \widehat{R_{\lambda }}\right\Vert _{\infty }
\end{eqnarray*}%
Now, we estimate $\left\Vert \widehat{R_{\lambda }}\right\Vert _{\infty }$.
Taking account of Proposition 2 (p.332 in \cite{stein3}), we note that
\[
\left\vert \widehat{R_{\lambda }}(\xi )\right\vert =\left\vert
\int\limits_{0}^{\left\vert \lambda \right\vert }e^{i(2^{m}sgn(\lambda
)\left\vert \lambda \right\vert ^{1-m}\sigma ^{m}-\xi \sigma )}d\sigma
\right\vert \leq \frac{C_{m}}{\left\vert \lambda \right\vert ^{\frac{1-m}{m}}%
}
\]%
where the constant $C_{m}$ does not depend on $\lambda $. Then for $%
Re(z)=-(n+\frac{1-m}{m})$, we have
\begin{eqnarray*}
\left\vert \upsilon _{N,z}(k,\lambda )\right\vert  &\leq &\frac{k!}{(k+n-1)!}%
\left\vert I_{1-z}(-\lambda )\phi _{N}(\lambda )\right\vert
2^{n-1}\left\vert \lambda \right\vert ^{-n}\left\Vert \widehat{F_{n,k}}\ast
\left( \widehat{G_{\lambda }}\ast \widehat{R_{\lambda }}\right) \right\Vert
_{\infty } \\
&\leq &\left\vert I_{1-z}(-\lambda )\right\vert \left\vert \phi _{N}(\lambda
)\right\vert 2^{n-1}\left\vert \lambda \right\vert ^{-n}\left( \int\limits_{%
\mathbb{R}}\frac{d\xi }{\left( \frac{1}{4}+\xi ^{2}\right) ^{\frac{n}{2}}}%
\right) \left\Vert \widehat{\eta _{0}}\right\Vert _{1}\frac{C_{m}}{%
\left\vert \lambda \right\vert ^{\frac{1-m}{m}}} \\
&\leq &C_{m}2^{n-1}\left\vert \Gamma \left( \frac{1-z}{2}\right) \right\vert
^{-1}\left\Vert H\right\Vert _{\infty }\left( \int\limits_{\mathbb{R}}\frac{%
d\xi }{\left( \frac{1}{4}+\xi ^{2}\right) ^{\frac{n}{2}}}\right) \left\Vert
\widehat{\eta _{0}}\right\Vert _{1}
\end{eqnarray*}%
Finally, by (\ref{L22}) it follows that, for $Re(z)=-\left( n+\frac{1-m}{m}%
\right) $
\[
\left\Vert U_{N,z}f\right\Vert _{L^{2}(\mathbb{H}^{n})}\leq \frac{C_{n,m}}{%
\left\vert \Gamma \left( \frac{1-z}{2}\right) \right\vert }\left\Vert
f\right\Vert _{L^{2}(\mathbb{H}^{n})}
\]%
is clear that the family $\left\{ U_{N,z}\right\} $ satisfies, on the strip $%
-\left( n+\frac{1-m}{m}\right) \leq Re(z)\leq 1$, the hypothesis of the
complex interpolation theorem. Thus $U_{N,0}$ is bounded from $L^{\frac{%
2(1+nm)}{2(1+mn)-m}}(\mathbb{H}^{n})$ into $L^{\frac{2(1+nm)}{m}}(\mathbb{H}%
^{n})$ uniformly on $N$, then doing $N$ tend to infinity we obtain that the
operator $U_{\nu _{m}}$ is bounded from $L^{\frac{2(1+nm)}{2(1+mn)-m}}(%
\mathbb{H}^{n})$ into $L^{\frac{2(1+nm)}{m}}(\mathbb{H}^{n})$, for $m,n\in
\mathbb{N}_{\geq 2}$.$\blacksquare $

\qquad

\textbf{Acknowledgment.} We express our thanks to the referee for his or her useful suggestions.

\qquad

\end{document}